\documentclass[10pt]{amsart}

\usepackage{amsmath, amssymb, amsthm, tikz}
\usepackage{hyperref}
 
\theoremstyle{plain}
\newtheorem{theorem}{Theorem}[section]
\newtheorem{proposition}[theorem]{Proposition}
\newtheorem{definition}[theorem]{Definition}
\newtheorem{example}[theorem]{Example}

\newtheorem{lemma}[theorem]{Lemma}

\newtheorem{remark}[theorem]{Remark}

\newcommand{\mc}{\mathcal}  
\newcommand{\Z}{\mathbb{Z}}   
\newcommand{\N}{\mathbb{N}}

\DeclareMathOperator{\FV}{FV}

\DeclareMathOperator{\kernel}{\mathsf{Kernel}}

\DeclareMathOperator{\diam}{\mathsf{diam}}

\begin{document}
\title[]{Finiteness of Homological Filling Functions}
\author{Joshua W. Fleming and Eduardo Mart\'inez-Pedroza}
 \address{Memorial University\\ St. John's, Newfoundland, Canada A1C 5S7}
 \email{emartinezped@mun.ca, jwf572@mun.ca}

\begin{abstract}
Let $G$ be a group. For any $\Z G$--module $M$ and any integer $d>0$, we define a function $\FV_{M}^{d+1}\colon \N \to \N\cup \{\infty\}$ generalizing the notion of $(d+1)$--dimensional filling function of a group.  We prove that this function takes only finite values if $M$ is of type $FP_{d+1}$ and $d>0$,  and remark that the asymptotic growth class of this function is an invariant of $M$.  In the particular case that $G$ is a group of type $FP_{d+1}$, our main result implies that its $(d+1)$-dimensional homological filling function takes only finite values, addressing a question from~\cite{HaMa14}. 
\end{abstract}

\maketitle

\section{Introduction}

For a contractible cellular complex $X$ and an integer $d>0$, the homological filling function $\FV_X^{d+1}\colon \N \to \N$  measures the difficulty of filling cellular $d$-cycles with $(d+1)$--chains, a precise definition is below. They are higher dimensional homological generalizations of isoperimetric functions.  For a group $G$ admitting a compact classifying space $K(G,1)$ with universal cover $X$, the equivalence growth rate of the function $\FV_X^{d+1}$ provides an invariant of the group.  The initial motivation of this work was to provide a direct argument that
$\FV_X^{d+1}$ takes only finite values for such complex $X$, addressing what the authors perceived as a gap in the literature. 
In this article we provide a self-contained proof based on the algebraic approach to define the homological filling functions from~\cite{HaMa14}, and on our way, we prove a more general result that defines a new collection of invariants for $\Z G$--modules.

\subsection*{The topological perspective.} We assume all spaces are combinatorial complexes and all maps are combinatorial, see for example~\cite[Ch.I Appendix]{BrHa99}.  A $G$-action on a complex $X$ is \emph{proper} if for all compact subcomplexes $K$ of $X$ the collection $\{g\in G \colon K \cap g(K) \neq \emptyset \}$ is finite.  The $G$-action is \emph{cocompact} if there is a compact subcomplex $K$ of $X$ such that the collection $\{gK\colon g\in G\}$ covers $X$.  For a complex $X$, the cellular $d$--dimensional chain group $C_d(X, \Z)$  is a free $\Z$-module with a natural $\ell_1$-norm induced by a basis formed by the collection of all $d$--dimensional cells of $X$, each cell with a chosen orientation from each pair of opposite orientations. This norm, denoted by $\|\cdot\|_1$, is the sum of the absolute value of the coefficients in the unique representation of the chain  as a linear combination of elements of the basis. Let $Z_d(X,\Z)$ denote the $\Z$--module of integral $d$--cycles, and $\partial_{d+1}\colon C_{d+1}(X,\Z) \to Z_d(X,\Z)$ be the boundary map.  The \emph{$(d+1)$--dimensional filling function of $X$} is the function $\FV_X^{d+1} \colon\N \to \N \cup \{\infty\}$ defined as 
\[\FV_X^{d+1} (k) = \sup \left \{ \ \| \gamma   \|_{\partial}  \colon \gamma \in Z_d(X, \Z), \ \| \gamma \|_1 \leq k \ \right \},\]
where 
\[ \| \gamma  \|_\partial  = \inf \left \{ \ \| \mu \|_1 \colon \mu \in C_{d+1}(X, \Z), \ \partial ( \mu ) = \gamma \ \right \},\]
where the supremum and infimum of the empty set are defined as zero and $\infty$ respectively. In words,  $\FV^{d+1}_{X}(k)$ is the most efficient upper bound on the size of fillings by $(d+1)$--chains of $d$--cycles of norm at most $k$. A complex $X$ is \emph{$d$-acyclic} if the reduced homology groups $\bar H_i (X, \Z)$ are trivial for $0\leq i\leq d$.
As mentioned above, the initial motivation of this work was to provide a proof of Theorem~\ref{thm:main} which the authors perceived as a gap in the literature.  The main contribution of this note is a generalization to an algebraic framework of the following statement, see Theorem~\ref{thm:main2}.

\begin{theorem}\label{thm:main}
Let $d$ be a positive integer and let $G$ be a group acting properly and cocompactly by cellular automorphisms on an $d$--acyclic complex $X$. Then $FV_X^{d+1}(m)$ is finite for all $m\in\N$. 
\end{theorem}

Theorem~\ref{thm:main} was known to hold in the following cases.
\begin{itemize}

\item  For $d=1$, it is a result of Gersten~\cite[Proposition 2.4]{Ge99}.

\item  For $d\geq 1$ and under the extra assumption that $G$ admits a combing, it follows from work of Epstein and Thurston~\cite[Theorem 10.3.6]{ECHLT92}; see also the recent work of Behrstock and Dru{\c{t}}u~\cite[Lemma 3.7]{BeDu15} and references therein. 

\item For $d\geq 3$, Hanlon and the second author observed in~\cite[Section 3.3]{HaMa14} that Theorem~\ref{thm:main} holds using results of Alonso, Pride and Wang~\cite{APW99} in conjuntion with an argument from Abrams, Brady, Dani and Young~\cite{ABDY13}.  The results in~\cite{APW99} rely on non-trivial machinery from homotopy theory.  The failure of the argument for $d=2$ relies on an application of the Hurewicz theorem, for details see~\cite[Section 3.3]{HaMa14} and the references therein.
\end{itemize}

Current results in the literature leave open the statement of Theorem~\ref{thm:main} for the case $d=2$. Our argument in this note proving Theorem~\ref{thm:main} does not rely on previous results, it is valid for all $d>0$, and it is elementary. The argument might be known to the experts, but to our knowledge does not appear in the literature, and this note fills this gap.  Let us sketch the argument from a topological perspective, for an an algebraic proof see Section~\ref{sec:tech}. 

\begin{proof}[Sketch of the proof of Theorem~\ref{thm:main}, from a topological perspective.]
Consider the combinatorial path metric on the $1$-skeleton of $X$, and 
for any $d$-cycle $\sigma$ (which is a formal finite sum of $d$-cells) define its diameter $\diam (\sigma)$ as the diameter of the set  consisting of vertices ($0$-cells) which are in the closure of at least one $d$-cell defining $\sigma$.  A $d$-cycle $\sigma$ is called \emph{connected} if the subcomplex of $X$ formed by taking the closure of the union of $d$-cells definining $\sigma$ is connected (and has no cut-points).

Let $m>0$.  Since $G$ acts properly and cocompactly on $X$,    
there is an integer $C\geq 0$ that bounds the diameter of any $d$-cell of $X$, and hence for any connected $d$-cycle $\sigma$
\[  \diam (\sigma ) \leq C \| \sigma \|_1.\]
From here, it follows that the induced $G$--action on the set of connected $d$-cycles with $\ell_1$-norm $\leq m$ has finitely many $G$--orbits.  Since $X$ is $d$-acyclic, $\|\sigma\|_\partial<\infty$ for each $d$-cycle $\sigma$.  Therefore, there is an integer $M=M(m)$, such that 
\[ \text{$\sigma$ is connected and  $\|\sigma\|_1\leq m$ } \quad  \Longrightarrow  \quad \|\sigma\|_\partial \leq M.\]

Let $\sigma$ be an arbitrary $d$-cycle with $\ell_1$-norm $\leq m$.  Then one shows that $\sigma$ can be decomposed as a sum of connected $d$-cycles $\sum_{i=1}^k \sigma_i$,  where $k \leq \|\sigma\|_1 = \sum_{i=1}^k \| \sigma_i \|_1$. Hence  
\[ \| \sigma \|_\partial \leq \sum_{i=1}^k \| \sigma_i \|_\partial \leq k \cdot M \leq m \cdot M. \]
Therefore $\FV_X^{d+1}(m) \leq m\cdot M<\infty$.
\end{proof}

\begin{remark}\label{rem:gpinvariance}
Under the assumptions of the Theorem~\ref{thm:main}, it is known that the growth rate of the function $\FV_X^{d+1}$ is a quasi-isometry invariant of the group $G$. This was first addressed by Fletcher in his PhD thesis~\cite[Theorem 2.1]{Fl98} under the   assumption that $X$ is the universal cover of a $K(G,1)$. In~\cite[Lemma 1]{Yo11}, Young provides a proof of the quasi-isometry invariance in the general context of Theorem~\ref{thm:main}. Notably, these works do not address that  these functions are finite. 
\end{remark}

\subsection*{The algebraic perspective, and our main result.}  Our main result is an algebraic analog of Theorem~\ref{thm:main}. Recall that for a group $G$,  a $\Z G$-module $M$ is of type $FP_n$ if there exists a partial resolution of $\Z G$-modules 
\begin{equation}\nonumber
P_n\xrightarrow{\varphi_n} P_{n-1} \xrightarrow{\varphi_{n-1}} \ldots \xrightarrow{\varphi_2} P_1 \xrightarrow{\varphi_1} P_0 \xrightarrow{} M \xrightarrow{} 0
\end{equation}
such that each $P_i$ is a finitely generated projective $\Z G$-module. For a $\Z G$-module $M$ of type $FP_{d+1}$ we define the {$(d+1)$--filling function $\FV^{d+1}_M$ of $M$}, see Definition~\ref{def:filling-module} in Section~2,   and prove the following result.

Recall that the \emph{growth rate class} of a function $\N \to \N$ is defined as follows.  Given two functions  $f,g: \mathbb{N}\to \mathbb{N}$, define the relation $f\preceq g$ if there is $C>0$ such that $f(n)\leq Cg(Cn+C)+Cn+C$ for all $n\in \mathbb{N}$; and let $f\sim g$ if both $f\preceq g$ and $g\preceq f$.  This yields an equivalence relation with the equivalence classes of a function $f$ called the \emph{growth rate class of $f$}.

\begin{theorem}\label{thm:main2} 
Let $M$ be a $\Z G$-module of type $FP_{d+1}$.
\begin{enumerate}
\item For all  positive integers $k$, $\FV_{M}^{d+1} (k ) < \infty$
\item The growth rate of the function $\FV_{M}^{d+1}\colon \N\to \N$ only depends on $M$.
\end{enumerate} 
\end{theorem}

This result provides a new collection of invariants for $\Z G$--modules that remains to be studied. The invariant is interesting even in the case that $M=\Z$ and $G$ is suitable.  In this case, the filling functions $\FV_\Z^{d+1}$ correspond to the filling invariants of the group $G$, usually denoted by $\FV_G^{d+1}$,  in the context of Theorem~\ref{thm:main} and Remark~\ref{rem:gpinvariance}. There are computations by Young~\cite{Yo16} in the case that $G$ corresponds to a discrete Heinserberg group  answering a conjecture of Gromov~\cite[Chapter 5]{Gr93}, estimations in the case that $G$ is the special linear group $SL (n,\Z)$ by Epstein and Thurston~\cite[Chapter 10]{ECHLT92},  and general results in the case that $G$ is a hyperbolic group by Gersten and Mineyev~\cite{Ge96, Mi00} among others. In~\cite[Remark 3.4]{HaMa14}, it was observed that there was no proof in the literature that if the $G$ is of type $FP_3$ (i.e. $\Z$ is of type $FP_3$ as a module over $\Z G$) then $\FV_G^3$ is finite valued; observe that this is a consequence of Theorem~\ref{thm:main2}.

This note contains a proof of the first statement of Theorem~\ref{thm:main2}. The proof of the second statement appears in~\cite[Theorem 3.5]{HaMa14} for the case that $M=\Z$, but the argument works verbatim for the general case.

\subsection*{Organization}
The rest of the paper is organized as follows: Section 2 contains  some preliminary definitions including the definition of $\FV_M^{d+1}$, the statement of the  main technical result of the article, Proposition~\ref{MT:1}, and arguments implying Theorems~\ref{thm:main} and~\ref{thm:main2}.  Section 3 is devoted to the proof of Proposition~\ref{MT:1}. The last section discusses some geometric examples illustrating some matters about Theorem~\ref{thm:main}. 

\subsection*{Acknowledgments}
The first author was funded by an Science Undergraduate Research Award (SURA) of Memorial University during a part of this project. The second author is funded by the Natural Sciences and Engineering Research Council of Canada (NSERC). We thank the referee for comments on the article, and Lana Mart\'inez-Aoki for assistance during this work.

\section{Main technical result and proof of the main theorems} \label{sec:tech}

Let $G$ be a group and let $S$ be a $G$-set. 
The set of all orbits of $S$ under the $G$-action is denoted by $S/G$.
The free abelian group $\Z[S]$ with $S$ as a free generating set can be made into $\Z G$-module that we shall call the permutation module on $S$. The $\Z$-basis $S$ induces a $G$-equivariant norm, called the \emph{$\ell_1$--norm},  given by $\left \| \   \sum_{s\in S} n_s s  \ \right \|_S =  \sum_{s \in S} | n_s |$, where $n_s \in \mathbb Z$.   

If the $G$-action on $S$ is free, then $\Z [S]$ is a free module over $\Z S$.  Conversely, if $F$ is a free $\Z G$--module with a chosen $\Z G-$basis $\{ \alpha_i | i\in I \}$, then $F$ is isomorphic to the permutation module $\Z [S]$ where $S=\left \{ g \alpha_i  \colon g \in G , i\in I \right \}$ with the natural $G$-action. In this case the $\Z G-$basis $\{ \alpha_i | i\in I \}$ of $F$ induces an $\ell_1$-norm as before.

\begin{definition}[Gersten's filling norms.] Let  $\eta\colon F \to M$ be a surjective  morphism of $\Z G$-modules where $F$ is finitely generated and free with a chosen finite $\Z G$--basis, the induced \emph{filling norm} on $M$ is defined by
\begin{equation}\nonumber
\|m\|_\eta = \min\{\|x\|_F : x\in F, \eta(x)=m \}.
\end{equation}
where $\|\cdot\|_F$ denotes the induced $\ell_1$-norm on $F$.
\end{definition}

\begin{remark}[Induced $\ell_1$-norms are filling norms]
Let $\Z[S]$ be a permutation $\Z G$-module such that $G$ acts freely on $S$ and the quotient $S/G$ is finite. Then  $\Z[S]$ is a finitely generated free $\Z G$--module and the $\ell_1$-norm $\|\cdot\|_S$ is a filling norm. This statement holds without the assumption that $G$ acts freely on $S$. Since we do not  use this fact, we leave its verification to the reader.
\end{remark}

\begin{definition}\label{def:fillm}
Let $\rho \colon \Z [S] \to \Z [T]$ be a morphism of permutation $\Z G$--modules such that the kernel $K = \ker\rho$ is finitely generated.  Let $\|\cdot\|_K$ denote a filling norm on $K$ and let $\|\cdot\|_S$ denote the $\ell_1$-norm on $\Z[S]$ induced by $S$. Define the function $\FV_{\rho} \colon \N \to \N\cup \{\infty\}$ as 
\begin{equation}\nonumber
\FV_{\rho} (n) = \sup\{ \| x \|_K | x\in K , \|x\|_S \leq n \}
\end{equation}
\end{definition}

\begin{proposition} \label{MT:1}
Let $\rho\colon \Z[S] \rightarrow \Z[T]$ be a morphism. Suppose that $S/G$ and $T/G$ are finite,  $T$ has finite $G$-stabilizers for all $t\in T$, and $\ker \rho$ is finitely generated. Then $\FV_\rho (n) < \infty$ for all $n\in \N$.
\end{proposition}

In the the remaining of this section, we deduce Theorems~\ref{thm:main} and~\ref{thm:main2} from Proposition~\ref{MT:1}. 

\begin{proof}[Proof of Theorem \ref{thm:main}]
Let $G$ be a group acting properly and compactly by cellular automorphisms on a $d$-connected complex $X$.  The free abelian groups $C_d (X)$ and $C_{d+1} (X)$ are permutation $\Z G$--modules over the $G$--sets of $d$--cells and $(d+1)$--cells of $X$, respectively.  Observe that the definitions of   $FV_X^{d+1}$ coincides with definition~\ref{def:fillm} of $FV_{\partial_{d}}$ for the boundary map $C_{d}(X) \xrightarrow{\partial_{d}} C_{d-1}(X)$. The proof concludes by verifying the hypothesis of 
Proposition~\ref{MT:1} for this morphism. 

Since the $G$--action on $X$ is cocompact, there are finitely many $G$-orbits of $d$--cells and $(d+1)$--cells; in particular $C_{d+1}(X)$ is a finitely generated $\Z G$--module.  Since $X$ is $d$--acyclic,  the sequence $C_{d+1}(X) \xrightarrow{\partial_{d+1}}  C_{d}(X) \xrightarrow{\partial_{d+1}} C_{d-1}(X)$ is exact and hence $\kernel (\partial_d )$ is a finitely generated $\Z G$--module. Since the $G$--action is proper, the stabilizer of each $d$-cell of $X$ is finite. This concludes the proof.
\end{proof}

\begin{definition} \label{def:filling-module}
Let $M$ be a $\Z G$-module of type $FP_{d+1}$. The \emph{$(d+1)$--filling function of $M$} is the   function 
\begin{equation}\nonumber
\FV_{M}^{d+1} \colon \N \to \N\cup\{\infty\}
\end{equation}
defined as follows. Let 
\begin{equation}\nonumber
P_{d+1}\xrightarrow{\varphi_{d+1}} P_{d} \xrightarrow{\varphi_d} \ldots \xrightarrow{\varphi_2} P_1 \xrightarrow{\varphi_1} P_0 \xrightarrow{} M \xrightarrow{}  0
\end{equation}
be a $FP_{d+1}$-resolution for $M$. Chose filling norms on $P_{d+1}$ and $P_d$ denoted by $\|\cdot \|_{P_{d+1}}$ and $\|\cdot\|_{P_{d}}$ respectively. Then 
\begin{equation}\nonumber
\FV_{M}^{d+1} (k) = \sup\{\|x\|_{\varphi_{d+1}} : x\in\ker\varphi_{d} \text{, } \|x\|_{P_{d}} \leq k\} 
\end{equation}
where
\begin{equation}\nonumber
\|x\|_{\varphi_{d+1}} = \min\{ \| y\|_{P_{d+1}} : y\in P_{d+1} , \varphi_{d+1} (y) = x \}
\end{equation}
\end{definition}

The proof of theorem \ref{thm:main2} uses the following lemma. 

\begin{lemma}[\cite{MR1324339},chapter 8, 4.3]\label{lem:freer}
A $\Z G$-module $M$ is of type $FP_d$  if and only if $M$ admits a partial resolution of free finitely generated $\Z G$-modules of the form
\begin{equation}\nonumber
F_{d+1}\xrightarrow{} F_{d} \xrightarrow{} \ldots \xrightarrow{} F_1 \xrightarrow{} F_0 \xrightarrow{} M \xrightarrow{} 0.
\end{equation}

\end{lemma}

\begin{proof}[Proof of Theorem \ref{thm:main2}]
Since $M$ is type $FP_{d+1}$, by Lemma~\ref{lem:freer},  there exists a partial resolution of free and finitely generated $\Z G$-modules 
\begin{equation}\nonumber
F_{d+1}\xrightarrow{\phi_{d+1}} F_{d} \xrightarrow{\phi_d} \ldots \xrightarrow{\phi_2} F_1 \xrightarrow{\phi_1} F_0 \xrightarrow{} M \xrightarrow{} 0.
\end{equation}
such that $\ker\phi_{n}$ is finitely generated for $n$ such that $d\geq n\geq 0$. Consider the finitely generated free modules $F_{d}$ and $F_{d-1}$ as permutation modules $\Z [S]$ and $\Z [T]$ respectively. Finite generation and freeness implies that we can assume that $G$ acts freely and with finitely many orbits on both $S$ and $T$. Since  the induced $\ell_1$-norms on $\Z [S]$ and $\Z [T]$ are in particular filling norms, 
the definition  of   $\FV_M^{d+1}$ coincides with definition~\ref{def:fillm} of $\FV_{\phi_{d+1}}$. 
Then the first statement of the theorem on finiteness of $\FV_M^{d+1}$ follows by applying Proposition~\ref{MT:1} to $\FV_{\phi_{d+1}}$.

The proof of the second statement that the growth rate of $\FV_M^{d+1}$ is independent of the choice of partial resolution and filling norms appears in~\cite[Theorem 3.5]{HaMa14} for the case that $M=\Z$ and $G$ is a group of type $FP_{d+1}$. The argument for arbitrary $M$ follows verbatim by replacing each ocurrence of $\Z$ by $M$. Let us remark that the heart of the argument is the fact that any two projective resolutions of $M$ are chain homotopy equivalent~\cite[pg.24, Theorem 7.]{MR1324339}. 
\end{proof}

\section{Finiteness}
This section contains the proof Proposition~\ref{MT:1}.
Let $S$ and $T$ be $G$-sets. For $x\in \Z[S]$ with $x=\sum_{s\in S} n_s s$, we denote by $\langle x,s\rangle $ the integer $n_s$. For $x \in \Z[T]$ and $t\in T$ we defined analogously $\langle x, t\rangle$.

\begin{definition}[$x$ is a part of $y$]
 Let $x,y \in \Z[S]$.  We say \emph{$x$ is a part of $y$}, denoted by $x\preceq_S y$, 
to mean that for each $s\in S$ if $\langle x,s\rangle >0$ then $\langle x, s\rangle \leq \langle y,s\rangle$, and if $\langle x,s\rangle <0$ then $\langle y,s\rangle \leq \langle x,s\rangle$. Note that this is equivalent to $\langle x,s\rangle\langle y,s\rangle \geq \langle x,s\rangle^{2}$ for all $s\in S$. 
\end{definition}

\begin{definition}[$S$-intersect]
For $x,y \in \Z[S]$, the $S$-intersection of $x$ and $y$ is defined as $x\cap_S y =\{s\in S : \langle x,s\rangle \langle y,s\rangle <0\}$.
\end{definition}

\begin{remark}
\label{Remark}
Let $x,y\in \Z[S]$. Then $\|x+y\|_S = \|x\|_S + \|y\|_S$ if and only if $x\cap_s y = \emptyset$. Indeed, 
\begin{equation}\nonumber
\|x+y\|_S  = \sum_{s\in S}|\langle x,s\rangle +\langle y,s\rangle | \leq \sum_{s\in S} |\langle x,s\rangle | + \sum_{s\in S}|\langle y, s \rangle | = \|x\|_S + \|y\|_S
\end{equation} 
with equality if and only if $\langle x,s\rangle$ and $\langle y,s \rangle$ have  the same sign for all $s\in S$.
\end{remark}

Throughout the rest of this section, let 
\[ \mc{D}_1 = S\cup \{-s | s\in S\}.\]
Furthermore, let $\rho\colon \Z[S] \to \Z[T]$ denote a morphism of permutation modules.

\begin{definition}[$\rho$-intersect]
A pair of elements $x,y\in \Z[S]$ have \emph{non-trivial  $\rho$-intersection}, denoted by  $x\cap_\rho y \neq \emptyset$, if there exists  $ x_1 , y_1\in\mc{D}_1$ such that   $\rho(x_1) \cap_T \rho(y_1) \neq \emptyset$ where $x_1 \preceq_S x$ and  $y_1\preceq_S y$.
\end{definition}

\begin{definition}[$\rho$-connected]
For each integer $n\geq 1$, let $\mc{D}_n$ be the collection of elements of $\Z[S]$ of the form $x=\sum_{i=1}^n x_i$ where each $x_i\in \mc{D}_1$ and for every $k<n$ the elements $\sum_{i=1}^k x_i$  and $x_{k+1}$ have trivial $S$-intersection and non-trivial $\rho$-intersection.
An element $x\in \Z[S]$ is   $\rho$\emph{-connected} if $x\in \mc{D}_n$ for some $n\geq 1$.
\end{definition}

\begin{remark}\label{rem:triviality}
For $x \in \Z[S]$,  $x\in \mc{D}_n$  if and only if $x$ is $\rho$-connected and $\|x\|\leq n$.
\end{remark}

\begin{lemma}
\label{LEM:4}
If $0\neq z\in\ker\rho$, then there exists $x$ such that 
\begin{enumerate}
\item $x \preceq_S  z$, in particular, $\| z-x \|_S < \|z\|_S$,
\item$x\in\ker\rho$, and 
\item $x$ is $\rho$-connected
\end{enumerate}
\begin{proof}
Let $0\neq z\in \ker\rho$ be an arbitrary element. Consider the set $\Omega=\{x\preceq_S z | x\neq 0\text{, }x\text{ is }\rho\text{-connected}\}$; this is a non-empty finite set partially ordered by $\preceq_S$. Let $x\in\Omega$ be a maximal element. 
We claim that $x\in\ker\rho$.
Suppose that $x\notin \ker\rho$. We have $\rho(x)$ and $\rho(z-x)$ are non-zero and satisfy
\begin{equation}\nonumber
 \rho(x) + \rho(z-x) = 0
\end{equation}
Since $\rho(x) \neq 0$ there exists $t\in T$ such that $\langle \rho(x), t \rangle \neq 0$. Therefore
\begin{equation}\nonumber
\langle \rho(z-x), t \rangle = -\langle \rho(x), t\rangle
\end{equation}
Since $\rho (z-x) \neq 0$, there exists $s\in S$ for which
\begin{equation}\nonumber
\langle z-x, s\rangle\langle \rho(s), t\rangle\langle \rho(z-x), t\rangle > 0
\end{equation}
This implies that
\begin{equation}\nonumber
\langle z-x,s  \rangle\langle \rho(s), t \rangle\langle \rho(x), t\rangle < 0
\end{equation}
Now define $\lambda = \frac{\langle z-x, s\rangle}{ |\langle z-x, s\rangle |}$. We show $x +\lambda s$ is $\rho$-connected. First observe that $x\cap_S \lambda s = \emptyset$ since $x\preceq_S z$ and $\lambda s \preceq_S z$. Moreover, note that $x\cap_\rho \lambda s \neq \emptyset$ since
\begin{equation}\nonumber
\langle \rho (x), t \rangle\langle \rho(\lambda s), t \rangle= \langle  \rho(x), t\rangle \langle  \rho (s), t\rangle \lambda < 0
\end{equation} 
Therefore $x + \lambda s$ is $\rho$-connected and $x \precneqq_S x + \lambda s \preceq_S z$. This contradicts the maximality of $x$ and therefore $x\in \ker\rho$.
\end{proof}
\end{lemma}

\begin{proposition}
\label{PROP:ML}
For all  nonzero $z\in \ker\rho$  there exists $\rho$-connected elements $x_1.\ldots ,x_n \in\ker\rho$ such that 
\begin{enumerate}
\item $z=x_1 +\dots +x_n$
\item $x_i \preceq_S z$ for each $i$
\end{enumerate}
\begin{proof}
Applying Lemma \ref{LEM:4} to $z\in\ker\rho$, there exists a $\rho$-connected element $x_1 \in \ker\rho$ such that $x_1 \preceq_S z$. If $z-x_1 \neq 0$ then there exists a $\rho$-connected element $x_2 \in\ker\rho$ such that $x_2 \preceq_S z-x_1 \prec_S z$. If $z-x_1-x_2 \neq 0$ then there exists a
$\rho$-connected element $x_3 \in\ker\rho$ such that $x_3 \preceq_S z-x_1-x_2 \prec z-x_1 \prec z$. This process must terminate for some positive integer $n$ since \[ \|z-x_1-\ldots -x_k\|>\|z-x_1-\ldots -x_k-x_{k+1}\|\geq 0\] if $z-x_1-\ldots -x_k \neq 0$. 
Hence we obtain $\rho$-connected elements $x_1,\ldots ,x_n \in\ker\rho$ such that  $x_i \preceq_S z$ for each $i$, and $z=x_1 +\dots +x_n$. 
\end{proof}
\end{proposition}

\begin{remark} \label{rem:trivial2}
For $x,y\in\Z[S]$ the relations $x\preceq_S y$, $x\cap_S y \neq \emptyset$, and $x\cap_\rho y\neq\emptyset$ are preserved by the $G$-action on $\Z[S]$. Thus, if $x\in\mc{D}_n$ and $g\in G$ then $gx\in\mc{D}_n$. It follows that $\mc{D}_n$ is a $G$-set.
\end{remark}

\begin{proposition}
\label{PROP:LFISF}
Suppose that $S$ and $T$ have finitely many $G$-orbits and each element  of $T$ has finite $G$-stabilizer. Then for every $n\geq 1$ the set $\mc{D}_n$ has finitely many $G$-orbits.
\end{proposition}

Before the proof of the Proposition \ref{PROP:LFISF}, we introduce the following lemmas. 

\begin{lemma}\label{lem:counting}
\label{LEM:Prev}
Suppose  $S$ has finitely many $G$-orbits,  and each element of $T$ has finite $G$-stabilizer. 
Then for every $t\in T$, the set $S(t) = \{s\in S| \langle  \rho(s), t \rangle \neq 0 \}$ is finite.
\end{lemma}

\begin{proof}
For any $t\in T$, $s\in S$, and $g\in G$ we have $\langle \rho (gs), gt \rangle = \langle  \rho (s), t \rangle$. For each $s\in S$, let $T(s) = \{ t\in T | \langle \rho (s), t \rangle \neq 0\}$. As $\rho$ is a morphism $T(s)$ is a finite set for all $s\in S$. Now, fix $t\in T$ and let $s_1,\ldots, s_m$ be representatives of $G$-orbits of $S$. Then
\begin{equation}\nonumber
\begin{split}
S(t) &= \bigcup_{i=1}^{m} \{gs_i | g\in G , \langle \rho(s_i), g^{-1}t \rangle \neq 0 \} \\ 
&= \bigcup_{i=1}^{m} \bigcup_{r \in T(s_i)} \{ gs_i | g\in G , g^{-1}t = r\} 
\end{split}
\end{equation}
Observe that the set $\{g\in G | g^{-1}t =r\}$ is in one-to-one correspondence with $G_t = \{g\in G | gt=t\}$. By assumption $G_t$ is finite and thus for each $i\in\{1,\ldots , m \}$ and $r\in T(s_i)$ the set $ \{ gs_i | g\in G , g^{-1}t = r\}$ is finite. Therefore, the set $S(t)$ is finite.
\end{proof}

\begin{lemma}
\label{LEM:Prev2}
Suppose $S$ has finitely many $G$-orbits and that $T$ has finite $G$-stabilizers for each $t\in T$. Then for all $n\in \Z_+$ and for all $y\in \mathcal{D}_n$ the set $\{x\in \mc{D}_1 : x\cap_\rho y \neq \emptyset\}$ is finite.
\end{lemma}

\begin{proof}
For $y\in \Z[S]$ denote by $\mc{D}_1(y)$ the set $\{x\in \mc{D}_1 : x\cap_\rho y \neq \emptyset\}$.
Let $y\in\mc{D}_{n}$. By definition,  $y = \sum_{i=1}^{n} x_i$ where each $x_i \in\mc{D}_1$   and for each $k<n$ the elements $\sum_{i=1}^{k} x_i$ and $x_{k+1}$ have trivial $S$-intersection and non-trivial $\rho$-intersection. It follows from the definition of $\rho$-intersect that 
\begin{equation}\nonumber
\mc{D}_1(y) =\{ x\in\mc{D}_1 : x\cap_\rho y \neq \emptyset\} = \bigcup_{i=1}^{n} \{ x\in \mc{D}_1 : x\cap_\rho x_{i} \neq \emptyset\} = \bigcup_{i=1}^{n} \mc{D}_1(x_i).
\end{equation}
Therefore, to conclude it is enough to show that $\mc{D}_1(s)$ is finite for every $s\in \mathcal{D}_1$.

Let $s\in \mc{D}_1$. Observe that
\[ \mc{D}_1(s) =   \bigcup_{t \in T } \{x\in\mc{D}_1 : \langle\rho (x) , t \rangle \langle \rho(s), t \rangle  < 0\} \subset \bigcup_{t \in T} \{x\in\mc{D}_1 :  \langle\rho (x) , t \rangle \langle \rho(s), t \rangle \neq 0 \}.\]
It is immediate $\{t\in T \colon \langle \rho(s), t \rangle \neq 0\}$ is finite. Hence the union on the right is over a collection with finitely many non-empty sets.  By Lemma~\ref{lem:counting}, for any $t\in T$ the set $\{x \in \mc{D}_1 \colon  \langle \rho(x) , t \rangle \neq 0\}$ is finite, and hence  $\{x\in\mc{D}_1 :  \langle\rho (x) , t \rangle \langle \rho(s), t \rangle \neq 0 \}$ is finite.  Therefore the expression on the right is the union of a finite collection of finite sets, and we conclude that $\mc{D}_1(s)$ is finite. 
\end{proof}

\begin{proof}[Proof of Proposition~\ref{PROP:LFISF}]
 We prove by induction on n.\\
For $n=1$ the result follows from the assumption that $S$ has finitely many $G$-orbits and the definition of $\mc{D}_1$. \\
Suppose $\mc{D}_n$ has finitely many $G$-orbits with representatives $y_1,\dots,y_\ell$. For each $1\leq k\leq \ell$, let $A_k$ be the collection of elements $A_k$ of $\mc{D}_1$ such that \begin{equation}\nonumber
y_k \cap_S z = \emptyset \text{ and }y_k\cap_\rho z \neq \emptyset.
\end{equation}
By Lemma \ref{LEM:Prev2}, the collection $A_k$ is finite. The proof concludes with the verification of the following claim. \\
\emph{Claim}: The set
\begin{equation} \nonumber
\{y_k + z : 1\leq k \leq \ell \text{ and } z\in A_k \} 
\end{equation}
is a collection of representatives of $G$-orbits of $\mc{D}_{n+1}$.\\
Let $x\in \mc{D}_{n+1}$. Then $x=\sum_{i=1}^{n+1} x_i$ where each $x_i\in \mc{D}_1$ and for every $k<n$ the elements $\sum_{i=1}^k x_i$  and $x_{k+1}$ have trivial $S$-intersection and non-trivial $\rho$-intersection. By definition $\sum_{i=1}^{n} x_i$ is in $\mc{D}_n$. Hence $ \sum_{i=1}^{n} x_i = gy_j$ for some $g\in G$ and some $1\leq j\leq \ell$. It follows that 
 $x= gy_j + x_{n+1}$ and therefore $g^{-1}x = y_j +g^{-1} x_{n+1}$.  By Remark~\ref{rem:trivial2}, we have that $z=g^{-1} x_{n+1}$ is an element of $A_j$. Therefore $x = gy_i + gz = g(y_i +z)$. This proves the claim. 
\end{proof} 

\begin{proof}[Proof of Proposition~\ref{MT:1}]
Let $K$ denote $\ker\rho$, and let $\|\cdot\|_K$ denote a chosen filling norm on $K$.  
By Proposition~\ref{PROP:LFISF}, for each positive integer $n$, the $G$-set $\cup_{i=1}^{n} \mathcal{D}_i$ has finitely many $G$-orbits. Therefore, for each $n \in \Z_+$ there is an integer $B_n$ such that for every $x\in \cup_{i=1}^{n}\mathcal{D}_i$, $\|x\|_K \leq B_n$.\\
Let $0\neq z\in K$ such that $\|z\|_S \leq n$. By Proposition \ref{PROP:ML}, 
there exists $\rho$-connected elements $x_1, \ldots , x_m \in K$ such that $m\leq n$, and 
$z=x_1 +\dots +x_m$, and each $x_i \prec z$. By Remark~\ref{rem:triviality}, each $x_i \in \mc{D}_n$. Therefore, by the triangle inequality, 
\[ \|z\|_K \leq \sum_{i=1}^{m} \|x_i\|_K \leq m\cdot B_n \leq n \cdot B_n.\]
This shows that  $\FV_\rho (n)\leq n\cdot B_n < \infty$.
\end{proof}

\begin{remark}
Observe that Proposition~\ref{MT:1} can be generalized as follows. Consider the sequence of modules $\ker\rho\rightarrow \Z[S] \xrightarrow{\rho}  \Z[T]$ where $|S/G|, |T/G| < \infty$ and $T$ has finite $G$-stabilizers for all $t\in T$. Let $\|\cdot\|_K$ be a $G$-invariant norm on $K$ then for all $n\in\N$
\begin{equation}\nonumber
\sup\{\|x\|_K | x\in K \text{  } \|x\|_S \leq n\} < \infty
\end{equation}
In particular, $K$ being finitely generated induces a filling norm which is $G$-invariant.
\end{remark}

\section{Examples}

\begin{figure}
\includegraphics{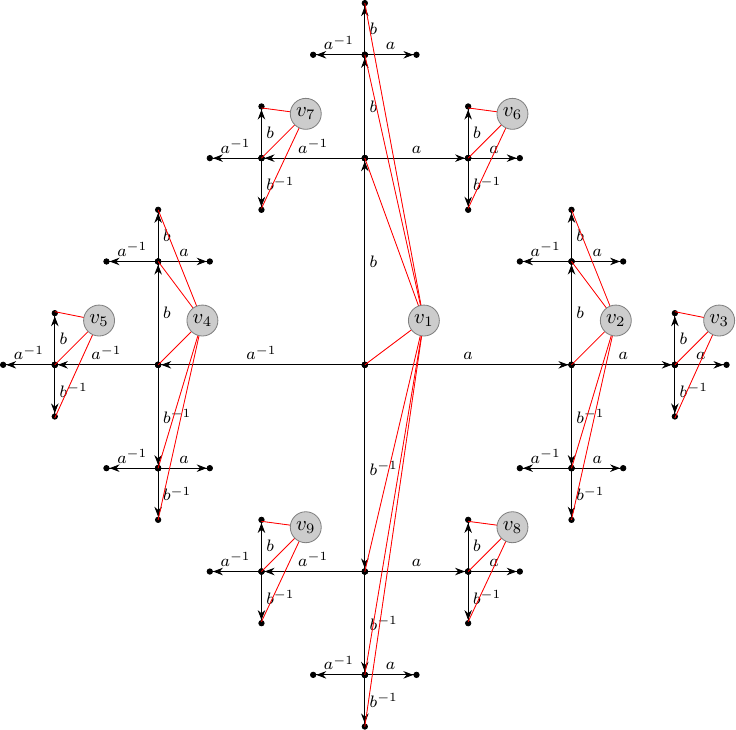}
\caption{Schematic the Coned-off Cayley graph $\hat \Gamma (G, P, S)$ where $G$ is the free group in two letters $S=\{a,b\}$ and $P$ is the cyclic subgroup $\langle b \rangle$} 
\end{figure}

A graph $\Gamma $ is called \emph{fine} if for every edge $e$ and each integer $n > 0$, the number of circuits of length at most $n$ which contain $e$ is finite.  By a circuit we mean a closed edge path that does not pass through a vertex more than once. The length of a circuit is defined as the number of edges.

\begin{theorem}\cite[Theorem 1.3]{Ma15}
\label{thm:3}
Let $X$ be a cocompact $G$-cell complex with finite stabilizers of 1-cells. The following two statements are equivalent:
\begin{enumerate}
\item $X$ has fine 1-skeleton and the homology group $H_1 (X, Z)$ is trivial,
\item $\FV_X (k) <\infty$ for any integer k.
\end{enumerate}
\end{theorem}

This result allows to exhibit the examples that contrast with Theorem~\ref{thm:main} as follows: 
\begin{itemize}
\item There is a group $G$ acting cocompactly, not propertly, and by cellular automorphisms on a simply-connected complex $X$ for which $\FV_X^{2}(m)$ is finite for all $m\in\N$. 

In particular, the converse of Theorem~\ref{thm:main} does not hold.

\item There is a group $G$ acting cocompactly by cellular automorphisms on a simply-connected complex $X$ for which $FV_X^{2}(m)$ is infinite for some $m\in\N$.  In particular, the properness assumption in Theorem~\ref{thm:main} can not be removed.
\end{itemize}

The two examples are based on the notion of coned-off Cayley complex.
We use the version from~\cite{GrMa08} which we briefly recall below; for another version see~\cite[Section 3]{Ma17}. 

\begin{figure}
\includegraphics{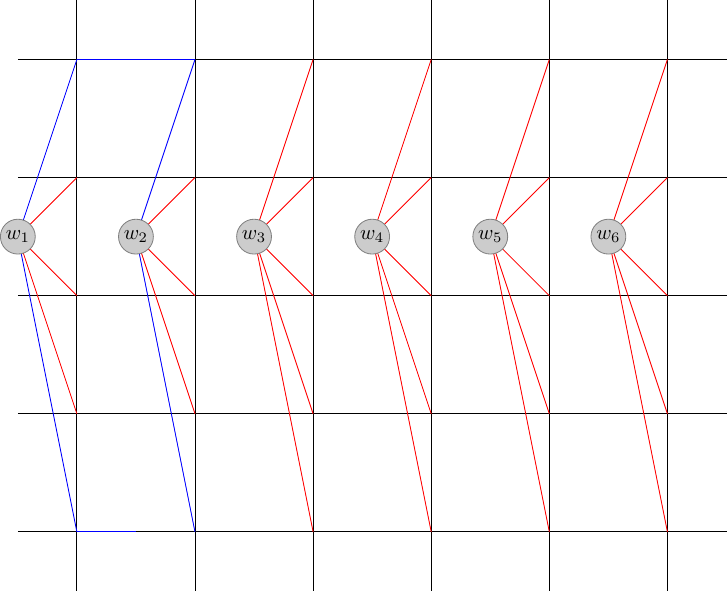}
\caption{Schematic the Coned-off Cayley graph $\hat \Gamma (G, P, S)$ where $G$ is the free abelian group in two letters $S=\{a,b\}$ and $P$ is the cyclic subgroup $\langle b \rangle$.}
\label{fig:torus} \end{figure}

Let $G$ be a group generated and let $P$ be a subgroup. The group $G$ is \emph{finitely generated relative to $P$} if there is a finite subset $S\subset G$ such that the natural map $F(S)\ast P \to G$ is surjective, where $F(S)$ denotes the free group on $S$, and $F(S)\ast P$ denotes the free product of $F(S)$ and $P$.  In this case $S$ is called a \emph{finite relative generating set of $G$ with respect to $P$}. 

Suppose that $S$ is a finite relative generating set of $G$ with respect to $P$. Without loss of generality assume that $S$ is closed under inverses.  The \emph{coned-off Cayley graph $\hat \Gamma = \hat \Gamma (G, P, S)$} is the graph with vertex set consisting of all elements of $G$ and all left cosets of $P$; the edge set is the collection of pairs $(g,gs)\in G\times G$ for $g\in G$ and $s\in S$, and pairs $(g, gP)$ for $g\in G$. Observe that the left action of $G$ on itself extends to a left action on $\hat \Gamma$. Vertices of $\hat \Gamma$ of the form $gP$ are called \emph{cone--vertices}.  Observe that the $G$-stabilizers of cone--vertices correspond to conjugates of $P$, in particular, if $P$ is infinite the action is not proper.  Moreover, the $G$--stabilizers of $1$-cells of $\hat \Gamma$ are trivial.  It is well known that the assumption that  $S$ is a relative generating implies that $\hat \Gamma$ is path--connected as a combinatorial complex, in fact, this is an equivalence as remarked in~\cite{HK08}. 

Under the assumptions, the group $G$ is \emph{finitely presented relative to $P$} if there is a finite subset $R\subset F(S) \ast P$ such that the kernel of the map $F(S)\ast P \to G$ is the smallest normal subgroup containing $R$. In this case, we say that
\begin{equation}\label{eq:fp} \langle S, P | R \rangle \end{equation}
is a finite relative presentation of $G$ with respect to $P$.  It is an exercise to show that if $G$ is finitely presented and $P$ is finitely generated, then $G$ is finitely presented relative to $P$. We refer the reader to~\cite{Os06} for an exposition on finite relative presentations.

Assume that $P$ is finitely generated, that~\eqref{eq:fp} is a finite relative presentation of $G$ with respect to $P$, and that $S\cap P$ is a generating set of $P$.  The \emph{coned-off Cayley complex $\hat C=\hat C(G, P, S, R)$} is the $2$-dimensional complex with $1$-skeleton the coned-off Cayley graph $\hat \Gamma (G,P,S)$  obtained by equivariantly attaching 2-cells  as follows. For  each word $r\in R$ correspond a loop in $\hat \Gamma$. Attach a $2$-cell with trivial stabilizer to each such loop, and extend in a manner equivariant under the $G$-action on $\hat \Gamma$. Similarly, for each $P\in \mc P$, for each generator in $s\in S\cap P$ and each $g\in G$ corresponds a loop in $\hat \Gamma$ of length three passing through the vertices $g, gs, gP$. Attach a $2$-cell with trivial stabilizer to each such loop, equivariantly under the $G$-action. The resulting $G$-complex $\hat C$ is simply-connected~\cite[Lemma 2.48]{GrMa08}, the $G$-action is cocompact by construction, and if $P$ is infinite the $G$-action is not proper.  Now we consider the $2$--dimensional filling function $FV^2_{\hat C}$ of $\hat C$.

\begin{example}\label{ex2}
Let $G$ be the free group of rank 2, let $S=\{a,b\}$ be a free generating set, and let $P$ be the cyclic subgroup generated by $b$. It is an observation that the coned-off Cayley graph $\hat \Gamma(G, P, S)$ is a fine graph and hence Theorem~\ref{thm:3} implies that $\FV^2_{\hat C}(m)<\infty$ for every $m\in \N$. Similar examplex can be constructed by considering relatively hyperbolic groups.
\end{example}

\begin{example}\label{ex1}
Let $G$ be the free abelian group of rank $2$, let $S=\{a,b\}$ be a generating set, and let $P$ be the cyclic subgroup generated by $b$. The coned-off Cayley graph $\hat \Gamma(G, P, S)$ is not fine since there are infinitely many circuits of length $6$ passing through the edge from $b$ to $P$. By Theorem~\ref{thm:3}, we have that $\FV^2_{\hat C}(m)=\infty$ for some $m\in \N$. In fact, one can verify that $\FV^2_{\hat C}(6)=\infty$.
\end{example}

\begin{remark} Theorem~\ref{thm:main} does not hold for $d=0$ in the natural setting of defining $\FV_{X}^1$  by taking  $Z_0(X,\Z)$ to be the the kernel of the augmentation map. Consider a finitely generated infinite group $G$ acting properly and cocompactly on a conneted graph $X$; for example, take $X$ the Cayley graph of $G$ with respect to a finite generating set. Then $X$ is infinite, and the formal difference $\gamma=b-a$ between two distinct vertices $a$ and $b$ of $X$ is a $0$-cycle for which $|\gamma|_\partial$ can be made  arbitrary large by taking $a$ and $b$ sufficiently appart; roughly speaking, a $1$-chain $\mu$ such that $\partial \mu = b-a$ contains a combinatorial edge path from $a$ to $b$ and hence $\|\mu\|_1$ is at least the length of the shortest edge path from $a$ to $b$. Henceforth $\FV^1_X(2)=\infty$ in this case. 
\end{remark}

\bibliographystyle{plain}

\end{document}